\documentclass[12pt,reqno]{amsart}

\usepackage{amsmath,amsfonts,amsthm,amssymb,amsxtra,bbm}
\usepackage{fourier,setspace,graphicx,xcolor}
\usepackage[colorlinks=true]{hyperref}
\hypersetup{urlcolor=blue, linkcolor=blue, citecolor=red, anchorcolor=blue]}
\newtheorem{theorem}{Theorem}%[section]

\newtheorem{lemma}[theorem]{Lemma}

\theoremstyle{definition}

\theoremstyle{remark}

\newtheorem{remark}[theorem]{Remark}

\renewcommand{\epsilon}{\varepsilon}

\newcommand{\N}{\mathbb{N}}

\renewcommand{\phi}{\varphi}
\newcommand{\R}{\mathbb{R}}

\newcommand{\Sph}{\mathbb{S}}

\def\bS{{\mathbb S}}
\def\N{{\mathbb N}}

\def\R{{\mathbb R}}

\def\rd{{\mathrm{d}}}

\def\cF{{\mathcal F}}

\def\cP{{\mathcal P}}

\newcommand{\ii}{\infty}

\renewcommand{\epsilon}{\varepsilon}

\newcommand{\norm}[1]{ \left| \! \left| #1 \right| \! \right| }

\title[Fast Diffusion leads to partial mass concentration]{Fast Diffusion leads to partial mass concentration in Keller-Segel type stationary solutions}

\author{J. A. Carrillo, M. G. Delgadino, R. L. Frank, M. Lewin}

\date{December 13, 2021}

\begin{document}

\maketitle
\begin{abstract}
We show that partial mass concentration can happen for stationary solutions of aggregation-diffusion equations with homogeneous attractive kernels in the fast diffusion range. More precisely, we prove that the free energy admits a radial global minimizer in the set of probability measures which may have part of its mass concentrated in a Dirac delta at a given point. In the case of the quartic interaction potential, we find the exact range of the diffusion exponent where concentration occurs in space dimensions $N\geq6$. We then provide numerical computations which suggest the occurrence of mass concentration in all dimensions $N\geq3$, for homogeneous interaction potentials with higher power.
\end{abstract}

\section{Introduction}
Nonlinear aggregation-diffusion equations of the form
\begin{equation}\label{FP}
\partial_t\rho=\Delta\rho^q+\,\nabla\cdot\left(\rho\,\nabla W\ast\rho\right)\,,
\end{equation}
are ubiquitous in continuous descriptions of populations, with applications in mathematical biology, gravitational collapse and statistical mechanics \cite{JL92,OS97,SC02,Ho03,BDP06,TBL06,SC08,HP09,BCL09,CHVY19,CCY19}. Here, $\rho(t)$ is a time-dependent probability measure over $\R^N$, $q>0$ is the diffusion exponent regulating if the diffusion is slow ($q>1$), linear ($q=1$) or fast ($0<q<1$) for small values of the density, while $W$ is the aggregation kernel describing some attraction between the agents of the population. In this work, $W$ will be taken equal to the homogeneous potential
$W_\lambda(x):=|x|^\lambda/\lambda$ with $\lambda>-N$. A positive $\lambda$ corresponds to a bounded interaction potential at the origin, whereas $-N<\lambda<0$ provide locally integrable, singular interaction potentials. At $\lambda=0$, the convention is to take $W_0(x)=\log|x|$ which, in dimension $N=2$, is the Newtonian potential.

The time-dependent equation (\ref{FP}) is the (formal) gradient flow of a free energy functional~\cite{CMV03,AGS,santambrogio2015optimal} defined for %(regular enough)
probability measures $\mu\in\cP(\R^N)$ by
\begin{align}
\cF[\mu]:=&\displaystyle
\frac{1}{q-1}\int_{\R^N}\mu_{\rm ac}(x)^q\,\rd x+\frac{1}{2\lambda}\,\iint_{\R^N\times\R^N} |x-y|^\lambda \,\rd\mu(x)\,\rd \mu(y)&q\ne 1,\label{freeenergy}
\end{align}
where $\mu_{\rm ac}$ is the absolutely continuous part of $\mu$ with respect to the Lebesgue measure. At $q = 1$, the first term is replaced by Boltzmann’s entropy, which is $\int_{\R^N}\mu_{\rm ac}(x)\log\mu_{\rm ac}(x)\,\rd x$ when the measure $\mu$ is absolutely continuous, and $+\infty$ for measures that are not absolutely continuous. This formal structure, which can be turned rigorous only for some particular values of the parameters \cite{CS18}, plays an essential role in the analysis of the dynamics in (\ref{FP}). In particular, the global minimizers of the free energy functional $\cF$ are the best candidates to be locally stable equilibria for~(\ref{FP}) and the large-time asymptotics of time-dependent solutions. 

% At $q=1$, the first term is replaced by Boltzmann's entropy $\int_{\R^N} \log\mu_{\rm ac}(x)\,\rd\mu(x)$. 

Nonlinear aggregation equations of the form (\ref{FP}) show very challenging phenomena, both with regard to their time behavior and the properties of their steady states. They have received lots of attention in the last 20 years, see for instance  \cite{JL92,OS97,SC02,Ho03,BDP06,TBL06,SC08,HP09,BCL09,CHVY19,CCY19,Pe1,Pe2} and the references therein. The form of the global minimizers depending on the parameters is quite rich, but by now it is well understood in the case of porous medium-like diffusion $q>1$ and linear diffusion $q=1$, see \cite{AucBea-71,LieOxf-80,Lions-81b,Lions84,JL92,DP04,BDP06,BCM08,BCC12,CD14,CCV15,GM18,CHMV18,CHVY19} for some references. When $q\geq 1$, for all possible values of $\lambda$ where global minimizers of the free energy exist, they are given by bounded probability densities.

We mention specifically the case of the classical parabolic--elliptic Keller--Segel (KS) system \cite{JL92,Ho03,DP04} for cell movement by chemotaxis, which is obtained by setting $N=2$, $q=1$ and $\lambda=0$.
Depending on the mass of the initial mass, it is known that solution to KS can blow up in finite time in any $L^p(\R^N)$, $p>1$. Several variations of KS that avoid blow up are present in the chemotaxis modelling literature, for a review see \cite{hillen2009user,bellomo2020chemotaxis}. We also mention the flux limited KS, in which the maximal aggregation velocity of the density is bounded, see \cite{perthame2019flux,bellomo2010multiscale}.

This work continues the investigation initiated in \cite{CDDFH19} where J.~Dolbeault, F.~Hoffmann and the first three authors of this paper dealt with the much less understood case of  $0<q<1$ and $\lambda>0$. It was shown in \cite{CDDFH19} that global minimizers of the free energy do not necessarily exist, even when $\cF$ is bounded-below, if we insist on requiring that $\mu$ is absolutely continuous with respect to the Lebesgue measure. However, minimizers in the form of measures exist when $\cF$ is bounded-below. By rearrangement, we can, in fact, restrict the minimization of $\cF$ to radially non-increasing probability distributions $\mu$ modulo translations~\cite{CHVY19}, and those can only have a singularity at the origin in the form of a Dirac delta. For these probability measures $\mu=M\delta_0+\rho$ with $\rho\in L^1(\R^N)$ a non-increasing radial function, the free energy becomes
\begin{align}\label{rhoM}
\mathcal F[\rho,M] \displaystyle= &- \frac{1}{1-q} \int_{\R^N} \rho(x)^q\,dx+ \frac{M}{\lambda} \int_{\R^N} |x|^\lambda\rho(x)\,\rd x \\
&+ \frac1{2\lambda} \iint_{\R^N\times\R^N} \rho(x) |x-y|^\lambda \rho(y)\,\rd x\,\rd y.\nonumber
\end{align}
It was proved in \cite{CDDFH19} that this free energy is bounded from below if and only if  $N/(N+\lambda)<q<1$, and in this case admits at least one global minimizer $\mu_*=M_*\delta_0+\rho_*$ with 
\begin{equation}
\rho_*\in L^q(\R^N)\cap L^1(\R^N)\cap L^1(\R^N,|x|^\lambda\,\rd x). 
 \label{eq:Lp_spaces_rho}
\end{equation}
The function $\rho_*$ is always supported on all of $\R^N$ due to the singularity of the term $\rho^q$ at zero, see \cite[Lemma 9]{CDDFH19}, and, if $M_*>0$, it diverges at the origin. We note parenthetically that the minimization of the free energy $\mathcal F$ is equivalent to finding the sharp constant in a reversed Hardy--Littlewood--Sobolev inequality.

Two important questions are whether these minimizers are unique modulo translations, and whether mass concentration occurs in the sense that $M_*>0$. These questions have been solved in~\cite{CDDFH19} for some values of the parameters, but only negatively concerning the occurrence of concentration. For instance, it was shown in~\cite{CDDFH19} that minimizers are always unique for $2\leq\lambda\leq4$ and that $M_*=0$ for 
\begin{itemize}
 \item all $\lambda>0$ and $\frac{N}{N+\lambda}<q<1$ if $N=1,2$,
 \item all $0<\lambda\leq 2+\frac{4}{N-2}$ and $\frac{N}{N+\lambda}<q<1$ if $N\geq3$.
\end{itemize}
Therefore, concentration never happens in dimensions $N=1,2$, nor in higher dimensions for too small values of the aggregation parameter $\lambda$. For example, concentration never happens at $\lambda=2$ in any dimension $N\geq1$. 

We have no intuitive explanation of why concentration cannot happen in dimensions $N=1,2$ but we notice that the same phenomenon occurs for other phase transitions in statistical physics, by the Mermin--Wagner theorem~\cite{MerWag-66,Simon-93}. This is for instance the case of Bose--Einstein condensation~\cite{Thirring}.

Our main goal in this work is to show that \textbf{concentration does indeed occur} in dimensions $N\geq3$, for large-enough values of $\lambda$. First, we look at the quartic case $\lambda=4$ which we can solve completely. We prove in Theorem~\ref{main} below that concentration happens for some (but not all) values of $q$ in dimension $N\ge 6$ but never occurs in dimensions $N\leq5$. However, it does occur in dimensions $N\in\{3,4,5\}$ at smaller values of $q$ where the free energy is unbounded from below, if we allow formal minimizers with an infinite free energy (see Remark~\ref{rmk:min_infinite}). Then, we provide numerical evidence that concentration does happen for regular minimizers at larger values of $\lambda$, starting with dimension $N=3$. 

Our numerical results detailed in Section~\ref{sec:numerics} suggest some interesting features of the model, which we are unfortunately not able to prove at the moment. The numerical evidence indicates that the concentration region satisfies some monotonicity in terms of our three parameters $N$, $\lambda$ and $q$. More precisely, if concentration occurs for some $(N,\lambda,q)$ then it should also occur for all $(N',\lambda',q')$ so that $N'\geq N$, $N'/(N'+\lambda)<q'\leq q$ and $\lambda'\geq\lambda$. Thus, for any fixed  $\lambda>0$ and $N\in\N$, the concentration region for $q$ should be an interval $\big(\frac{N}{N+\lambda},q_N(\lambda)\big)$, with $q_N(\lambda)$ a non-decreasing function. For instance, we think that the interval starts to be non-empty for $\lambda$ slightly above $8$ in dimension $N=3$. Proving these observations seems challenging. 

Our theoretical and numerical results are both based on the associated (first order) Euler-Lagrange equation for the absolutely continuous part $\rho_*$:
\begin{equation}
\frac{q}{1-q}\rho_*^{q-1}= \rho_*\ast|\cdot|^\lambda+M_*|x|^\lambda-\int_{\R^N}|y|^\lambda\rho_*(y)\,\rd y\;+L,
 \label{eq:EL1}
\end{equation}
where $\frac{q}{1-q}\rho_*^{q-1}(0)=L\ge 0$ is a Lagrange multiplier associated with the mass constraint 
\begin{equation}\label{mass}
\int_{\R^N} \rho(x)\,\rd x + M = 1.
\end{equation}
A useful fact will be that~(\ref{eq:EL1}) can be highly simplified when $\lambda$ is an even integer. In this case we can expand $|x-y|^{2n}$ in terms of a polynomial in $x$ and $y$ and, using the radial symmetry of $\rho_*$, the convolution $\rho_*\ast|\cdot|^\lambda$ becomes a simple polynomial in $|x|^2$. Thus~(\ref{eq:EL1}) can be turned into a simpler equation for these finitely many polynomial coefficients. This is how we will be able to solve completely the particular case $\lambda=4$. This simplification will also be used to design a rather precise numerical algorithm in the case $\lambda\in 2\N$, and to deal with the other values of $\lambda$. 

We conclude this introduction with more comments about the meaning of our findings in light of the time-dependent equation~(\ref{FP}). Any minimizer $\mu_*=M_*\delta_0+\rho_*$ for the free energy (\ref{freeenergy}) is a stationary solution of (\ref{FP}), in the sense that 
$$
\Delta\rho_*^q+\,\nabla\cdot\left(\mu_*\,\nabla W_\lambda\ast\mu_*\right)=0
$$
as distributions, due to the properties~(\ref{eq:Lp_spaces_rho}) proved in~\cite{CDDFH19}. We point out that the range $0<q<1$ is usually called ``fast diffusion'' corresponding to the faster diffusion (than the heat equation) for small values of the density while the diffusion is in fact slower for large values of the density, where concentration happens. Therefore, when $q$ gets smaller, the diffusion of a Dirac delta at the origin is weaker. Even more, as proved in \cite{BF83} by Br\'ezis and Friedman, an initial Dirac delta is in fact a kind of stationary solution to the fast diffusion equation for $0<q<\frac{N-2}{N}$. More precisely, they showed that approximating an initial Dirac delta at the origin by mollifiers and sending the regularization parameter to zero was not leading to a source-type strong $L^1$-solution (Barenblatt-type solution) as opposed to the case  $q>\frac{N-2}{N}$. In fact, they proved that the Dirac delta stays ``stable'' for all times \cite[Thm.~8]{BF83}. It is an open problem to give sense to the evolution problem~(\ref{FP}) with general probability measures as initial datum. The Wasserstein gradient flow of the suitably defined extension of the free energy (\ref{freeenergy}) to measures as in \cite[Sect.~5]{CDDFH19} is the natural candidate.

Our results show that the combined effect of the small diffusion at large densities with the attraction due to the potential $W_\lambda$ can lead to a partial concentration of mass at the origin in the stationary case. Understanding the well-posedness and long time behavior of the evolution equation~(\ref{FP}) is essentially open in this regime of parameters. The possible concentration of mass happening in finite or infinite time for its evolution is just one of the many challenging open questions. The long time behavior has recently been investigated in the preprint~\cite{CaoLi-20}, but for $q$ so close to 1 that no concentration can happen.

\medskip

The paper is organized as follows. The next section is devoted to the special case $\lambda=4$ whereas in Section~\ref{sec:numerics} we present our numerical simulations for other values of $\lambda$, looking first at the special case of even integers and then general values.

%%%%%%%%%%%%%%%%%%%%%%%%%%%%%%%%%%%%%%%%%%%%%%%%%%%%%
%%%%%%%%%%%%%%%%%%%%%%%%%%%%%%%%%%%%%%%%%%%%%%%%%%%%%
\section{The quartic case $\lambda=4$}\label{sec:lambda4}
%%%%%%%%%%%%%%%%%%%%%%%%%%%%%%%%%%%%%%%%%%%%%%%%%%%%%
%%%%%%%%%%%%%%%%%%%%%%%%%%%%%%%%%%%%%%%%%%%%%%%%%%%%%
In the case $\lambda=4$, we are reduced to the regime 
$$
\frac{N}{N+4} < q <\frac{N}{N+2} \,,
$$
since it was already proved in \cite[Prop.~14]{CDDFH19} that concentration does not happen for larger values of $q$ and that the free energy is not bounded below for smaller values. We know from \cite[Prop.~20]{CDDFH19} that there is a minimizer $(\rho_*,M_*)$ and by \cite[Thm.~27]{CDDFH19} that this minimizer is unique up to translations. The question is whether $M_*=0$ or not. We define
$$
\boxed{q_N(4) := \frac{N-2}{N+2}\left(1+\frac{4}{3N}\right)}
$$
which will be proved to be the critical exponent at $\lambda=4$ in the next statement. For all $N\geq 6$, one has
$$
\frac{N}{N+4} < q_N(4) <\frac{N-2}{N}<\frac{N}{N+2},
$$
whereas $q_N(4)<\frac{N}{N+4}$ for $N\leq 5$. Note that $q_N(4)$ is increasing with the dimension $N$. 

\begin{theorem}\label{main}
	Let $\lambda=4$ and $\frac{N}{N+4} < q <\frac{N}{N+2}$. Then $\cF$ in~(\ref{rhoM}) admits a unique minimizer $(\rho_*,M_*)$ with $\rho_*$ a radial function satisfying~(\ref{eq:Lp_spaces_rho}). 
	\begin{enumerate}
		\item[(a)] If $N\leq 5$, then $M_*=0$.
		\item[(b)] If $N\geq 6$ and $q<q_N(4)$, then $\displaystyle M_*=\frac{3}{2} \, \frac{q_N(4)-q}{\frac{N-2}{N}-q} >0$.
		\item[(c)] If $N\geq 6$ and $q\geq q_N(4)$, then $M_*=0$.
	\end{enumerate}
\end{theorem}

In dimensions $N\geq 6$ with $q\leq q_N(4)$, we are able to compute the unique minimizer explicitly. It is given by
\begin{equation}
\label{eq:optimizer}
\rho_*(x) = \left(\frac{q}{1-q}\right)^{\frac1{1-q}}\left(|x|^4 + B_*|x|^2\right)^{-\frac1{1-q}}
\end{equation}
for a constant $B_*>0$ given in the proof of Lemma~\ref{solutions}. Note that $\rho_*$ diverges at the origin, as it should~\cite{CDDFH19}. Note also that the concentrated mass $M_*$ is decreasing with respect to $q$. One could also compute $\mathcal F[\rho_*,M_*]$ explicitly for $q\leq q_N(4)$.
See Remark~\ref{rmk:min_infinite} below for an interpretation of what is happening in dimensions $N\in\{3,4,5\}$. 

\subsection*{Proof of Theorem~\ref{main}}
First, we recall that any minimizer of the form $\mu_*=M_*\delta_0+\rho_*$ solves the nonlinear equation~(\ref{eq:EL1}). Using 
$$|x-y|^4=|x|^4 + |y|^4 + 4(x\cdot y)^2 + 2|x|^2|y|^2 - 4|x|^2 x\cdot y - 4|y|^2 x\cdot y$$
and the fact that $\rho_*$ is radial, we can express the convolution in the form
\begin{align}
 \int_{\R^N} |x-y|^4\rho_*(y)\,\rd y = &\,|x|^4 \int_{\R^N} \rho_*(y)\,\rd y + \int_{\R^N} |y|^4\rho_*(y)\,\rd y\nonumber\\
 &\quad+ \!\left( 2 + \frac{4}{N}\right) |x|^2\! \int_{\R^N} |y|^2 \rho_*(y)\,\rd y \,.
 \label{eq:convolution_radial}
\end{align}
Thus, using the mass constraint~(\ref{mass}), the Euler-Lagrange equation~(\ref{eq:EL1}) for $\rho_*$ can be rewritten in the form
\begin{equation}
\begin{cases}
\displaystyle \frac{q}{1-q}\rho_*(x)^{q-1}= |x|^4+B|x|^2+L,\\[3mm]
\displaystyle B=\left( 2 + \frac{4}{N}\right) \int_{\R^N} |y|^2 \rho_*(y)\,\rd y,
\end{cases}
 \label{eq:EL_lambda_4}
\end{equation}
where we recall that $L\geq0$ is an unknown Lagrange multiplier. Our goal is to find the values of $L$ and $B$ in~(\ref{eq:EL_lambda_4}). Of course we then have $M_*=1-\int_{\R^N}\rho_*$.  

The idea of the proof is to look at all the possible solutions of~(\ref{eq:EL_lambda_4}) parametrized by $L$. Plugging the first formula into the second, we obtain a simple nonlinear equation for $B$ which we show admits a \emph{unique solution} $B(L)$ for any $L\geq0$. A scaling argument will also give us an exact expression of $B(0)$. Next, we show that the mass $m(L)$ of this solution is strictly decreasing with~$L$. Thus, we have two possibilities either $m(0)\geq1$ or $m(0)<1$. If $m(0)\geq1$, then there exists a unique $L_*\geq0$ such that $m(L_*)=1$, and the minimizer is given by $(\rho_*,0)$, where $\rho_*$ is the solution of~(\ref{eq:EL_lambda_4}) with this $L_*$. This conclusion follows from the uniqueness of critical points, which is shown by a convexity argument. Alternatively, if $m(0)<1$, then there exist no solution of~(\ref{eq:EL_lambda_4}) with mass $1$ and thus there must be concentration. From~\cite[Prop.~14]{CDDFH19} we then know that $L=0$, that is, $\rho_*$ is given by the unique solution with $L=0$ and $B=B(0)$ and we have $M_*=1-m(0)>0$. 

The following contains the main properties of solutions of~(\ref{eq:EL_lambda_4}) that we need for the proof. 

\begin{lemma}\label{solutions}
	Let $\lambda=4$ and $\max\left(0,\frac{N-2}{N+2}\right)<q<\frac{N}{N+2}$.
	\begin{enumerate}
		\item[(a)] There is a unique differentiable function $B:[0,\infty)\to (0,\infty)$ such that for each $L\geq 0$,  
		$$
		\rho_L(x):=\left(\frac{q}{1-q}\right)^{\frac1{1-q}}\left(|x|^4+B(L)|x|^2+L\right)^{-\frac1{1-q}}
		$$
		satisfies
		$$\left( 2 + \frac{4}{N}\right) \int_{\R^N} |y|^2 \rho_L(y)\,\rd y=B(L).$$
		\item[(b)] The mass $m(L):=\int_{\R^N}\rho_L(x)\,\rd x\in(0,+\infty]$ is continuous and strictly decreasing with respect to $L\geq 0$. It converges to 0 when $L\to\ii$.
		\item[(c)] At $L=0$ we have 
		\begin{equation}
		m(0)=\begin{cases}
		\frac12\, \frac{q-\frac{N-2}{N+2}}{\frac{N-2}N - q} &\text{if $q<\frac{N-2}N$,}\\[3mm]
		+\ii & \text{otherwise.}
		\end{cases}
		 \label{eq:formula_m(0)}
		\end{equation}
		Thus, if $q<\frac{N-2}{N}$, then $m(0) \leq 1$ if and only if $q\leq q_N(4)$. 
	\end{enumerate}
\end{lemma}

\begin{proof}[Proof of Lemma~\ref{solutions}]
\textit{Part (a).} 
Let us define
\begin{equation}
\rho_{B,L}(x)=\left(\frac{q}{1-q}\right)^{\frac1{1-q}}\left(|x|^4+B|x|^2+L\right)^{-\frac1{1-q}} 
 \label{eq:def_rho_B_L}
\end{equation}
for all $B,L\geq 0$ and consider the function $F_L(B)$ given by
\begin{align}
F_L(B):=\left(\frac{q}{1-q}\right)^{\frac1{1-q}}|\bS^{N-1}|\,\int_{0}^\ii \frac{r^{N+1}\rd r}{\left(r^4+Br^2+L\right)^{\frac1{1-q}}}= \int_{\R^N} |y|^2\rho_{B,L}(y)\;\rd y .
 \label{eq:B11a}
\end{align}
This function is well defined for all $L\geq0$ and $B>0$ under our assumptions on $q$. It is strictly decreasing as a function of $B$ and $L$ separately. Moreover, for every given $L\geq 0$ the limit as $B\to 0^+$ is positive (infinite if $L=0$) and the limit $B\to \infty$ is zero.  Therefore, for every $L\geq0$, there exists a unique $B(L)>0$ such that
\begin{align}
B(L)=\kappa\,F_L\big(B(L)\big)\qquad \mbox{with } \kappa:=2+\frac4N .
 \label{eq:B11}
\end{align}
From the monotonicity of $F_L$, we have that $B$ is decreasing with $L$. From the implicit function theorem, $B$ is in fact a smooth function of $L$. Using $r^4+B(L)r^2+L\geq r^4+L$ and scaling out $L$ we find 
\begin{align*}
B(L)&\leq \kappa \left(\frac{q}{1-q}\right)^{\frac1{1-q}}|\bS^{N-1}|\,\int_{0}^\ii \frac{r^{N+1}\rd r}{\left(r^4+L\right)^{\frac1{1-q}}}\\
&=\kappa \left(\frac{q}{1-q}\right)^{\frac1{1-q}}|\bS^{N-1}|L^{\frac{N+2}{4}-\frac1{1-q}}\,\int_{0}^\ii \frac{r^{N+1}\rd r}{\left(r^4+1\right)^{\frac1{1-q}}}.
\end{align*}
Under our assumption on $q$ we have $\frac{N+2}{4}<\frac1{1-q}$ and thus $B(L)\to0$ when $L\to\ii$. We will compute the exact value of $B(0)$ below. 

\smallskip

\noindent \textit{Part (b).} Now we show that the mass $m(L)$ of $\rho_L$ is decreasing in $L$. Similarly as above, we write the integral in radial coordinates to obtain
\begin{equation}
m(L)=\left(\frac{q}{1-q}\right)^{\frac1{1-q}}\!\!\!|\bS^{N-1}|\int_{0}^\ii \!\!\! \frac{r^{N-1}\rd r}{\left(r^4+B(L)r^2+L\right)^{\frac1{1-q}}}.
\label{eq:formula_m_L}
\end{equation}
The integral converges for all $L>0$, but not necessarily for $L=0$. The same estimate as for $B(L)$ provides 
$$m(L)\leq \left(\frac{q}{1-q}\right)^{\frac1{1-q}}|\bS^{N-1}|L^{\frac{N}{4}-\frac1{1-q}}\int_{0}^\ii \frac{r^{N-1}\rd r}{\left(r^4+1\right)^{\frac1{1-q}}}$$
and shows that $m(L)\to0$ when $L\to\ii$. Then, taking a derivative we obtain
$$m'(L)=-\int_{0}^\ii (1+B'(L)r^2)\phi(r)\,\rd r,$$
where we have introduced 
$$\phi(r):=\left(\frac{q}{1-q}\right)^{\frac1{1-q}}\frac{|\bS^{N-1}|}{1-q}r^{d-1}\left(r^4+B(L)r^2+L\right)^{-\frac1{1-q}-1}$$
to simplify the notation. On the other hand, we have from~(\ref{eq:B11})
$$ B'(L)=-\kappa\int_{0}^\ii r^2(1+B'(L)r^2)\phi(r)\,\rd r.$$
Hence
$$B'(L)=-\frac{\kappa\int_{0}^\ii r^2\phi(r)\,\rd r}{1+\kappa\int_{0}^\ii r^4\phi(r)\,\rd r}$$
and thus, 
\begin{align*}
m'(L)=-\int_{0}^\ii \phi(r)\,\rd r+\frac{\kappa\left(\int_{0}^\ii r^2\phi(r)\,\rd r\right)^2}{1+\kappa\int_{0}^\ii r^4\phi(r)\,\rd r}\leq -\frac{\int_{0}^\ii \phi(r)\,\rd r}{1+\kappa\int_{0}^\ii r^4\phi(r)\,\rd r}<0, 
\end{align*}
since 
$$
\left(\int_{0}^\ii r^2\phi(r)\,\rd r\right)^2 \leq
\int_{0}^\ii \phi(r)\,\rd r \int_{0}^\ii r^4\phi(r)\,\rd r
$$
by the Cauchy-Schwarz inequality. This proves our claim that the mass is a decreasing function of $L>0$.
	
\smallskip
	
\noindent \textit{Part (c).} For $L=0$, we notice that the right side of \eqref{eq:B11} is homogeneous in $B(0)$. More explicitly, making the change of variables $r=\sqrt{B(0)}\bar r$, we see that \eqref{eq:B11} can be expressed as
$$
B(0) =\left(2+\frac{4}{N}\right)F_0(B_0)= \left( 2 + \frac{4}{N}\right) c_{N,q}\, B(0)^{\frac{N+2}{2}-\frac{2}{1-q}} \left(\frac{q}{1-q}\right)^{\frac{1}{1-q}},
$$
or equivalently
\begin{equation}
 B(0)^{\frac2{1-q}-\frac{N}2} = \left( 2 + \frac{4}{N}\right) c_{N,q}\left(\frac{q}{1-q}\right)^{\frac{1}{1-q}},
 \label{eq:formula_B_0}
\end{equation}
with
$$
c_{N,q} = |\Sph^{N-1}| \int_0^\infty \frac{r^{N+1}}{(r^4 + r^2)^{\frac{1}{1-q}}}\,\rd r \,,
$$
which is finite under our assumptions on $q$. Next, we turn to the mass. Since $B(L)\to B(0)>0$ as $L\to 0^+$, we have 
$$\lim_{L\to0^+}m(L)=\left(\frac{q}{1-q}\right)^{\frac1{1-q}}|\bS^{N-1}|\int_{0}^\ii \frac{r^{N-1}\rd r}{\left(r^4+B(0)r^2\right)^{\frac1{1-q}}}.$$
This is infinite for $q\geq \frac{N-2}N$ due to the singularity at the origin. For $q<\frac{N-2}N$ we can compute the explicit value of $m(0)$ in~(\ref{eq:formula_m(0)}), using the formula~(\ref{eq:formula_B_0}) of $B(0)$. By scaling we have, this time,
\begin{align*}
m (0)=\int_{\R^N} \rho_0(y)\,\rd y & = c_{N,q}' \,  B(0)^{-\frac{N}{2}+\frac{2}{1-q}}\left(\frac{q}{1-q}\right)^{\frac{1}{1-q}}
\end{align*}
with
$$
c_{N,q}' = |\Sph^{N-1}| \int_0^\infty \frac{r^{N-1}}{(r^4 + r^2)^{\frac{1}{1-q}}}\,\rd r.
$$
Inserting~(\ref{eq:formula_B_0}) into this expression, we obtain
$$
m(0) = \frac{c_{N,q}'}{c_{N,q}} \, \frac{N}{2(N+2)} \,.
$$
With the change of variables $t=r^2$ we can write
\begin{align*}
c_{N,q} &= \frac12 \, |\Sph^{N-1}| \int_0^\infty t^{\frac{N}2-\frac{1}{1-q}} (1+t)^{-\frac{1}{1-q}}\,\rd t
\end{align*}
and
$$
c_{N,q}' = \frac12 \, |\Sph^{N-1}| \int_0^\infty t^{\frac{N-2}2-\frac{1}{1-q}} (1+t)^{-\frac{1}{1-q}}\,\rd t.
$$
Thus, by beta and gamma function identities,
$$
\frac{c_{N,q}'}{c_{N,q}} = \frac{\Gamma(\tfrac{N}{2}-\tfrac{1}{1-q})\,\Gamma(\tfrac{2}{1-q}-\tfrac{N}{2})}{\Gamma(\tfrac{N+2}{2}-\tfrac{1}{1-q})\,\Gamma(\tfrac{2}{1-q}-\tfrac{N+2}{2})}= \frac{N+2}{N}\,\frac{q-\frac{N-2}{N+2}}{\frac{N-2}N - q},
$$
which proves (\ref{eq:formula_m(0)}) and concludes the proof of the Lemma~\ref{solutions}.
\end{proof}

Note that Lemma~\ref{solutions} covers a larger range for the exponent $q$ since $\frac{N-2}{N+2}<\frac{N}{N+4}$, see Remark~\ref{rmk:min_infinite}. However, the free energy $\cF$ is bounded-below only for $q>\frac{N}{N+4}$, which we assume from now on. For the convenience of the reader we provide a self-contained proof of Theorem~\ref{main} using Lemma~\ref{solutions}, which does not use any material from~\cite{CDDFH19} and solely relies on the convexity of $\cF$ in the cone of radial probability measures, in the spirit of Lopes' work \cite{lopes2017uniqueness}.

There are two situations. When $q\geq q_N(4)$, Lemma~\ref{solutions} implies that there is a unique $L_*\geq0$ so that $m(L_*)=1$. In this case we define $(\rho_*,M_*):=(\rho_{L_*},0)$. On the other hand, when $q<q_N(4)$ (which can only happen in dimensions $N\geq6$), then the equation $m(L)=1$ admits no solution. In this case we choose $L_*=0$ and set $(\rho_*,M_*):=(\rho_0,1-\int_{\R^N}\rho_0)$, which satisfies $M_*>0$. The choice $L_*=0$ is dictated by~\cite[Prop.~14]{CDDFH19} but we will see below that this follows from the first order Euler-Lagrange condition on $M_*$, which we have not yet used in the argument.  

We claim that the so-defined $(\rho_*,M_*)$ is the unique minimizer of $\cF$, which will conclude the proof. To prove this claim, we notice that for $\mu=M\delta_0+\rho$ with $\rho$ a radial function, we can use~(\ref{eq:convolution_radial}) to write
\begin{equation}
\cF[\rho,M]=\int_{\R^N}|x|^4\rho(x)\,\rd x+\left(1+\frac{2}{N}\right)\left(\int_{\R^N}|x|^2\rho(x)\,\rd x\right)^2-\frac1{1-q}\int_{\R^N}\rho(x)^q\rd x ,
\label{eq:free_energy_lambda_4_radial}
\end{equation}
From this, it is apparent that $\cF$ is strictly convex in $\rho$. Let us now prove that 
$$
\cF[\rho,M]\geq \cF[\rho_*,M_*]\qquad\mbox{for every   $\;\displaystyle\int_{\R^N}\rho+M=1$},
$$ 
with equality if and only if $(\rho,M)=(\rho_*,M_*)$.
Using~(\ref{eq:free_energy_lambda_4_radial}) and the definition~(\ref{eq:def_rho_B_L}) of $\rho_*$ we obtain after a calculation
\begin{align}
 \cF[\rho,M]-\cF[\rho_*,M_*]=&L_*\int_{\R^N}(\rho_*-\rho)(x)\,\rd x
 +\left(1+\frac{2}{N}\right)\left(\int_{\R^N}|x|^2(\rho-\rho_*)(x)\,\rd x\right)^2\nonumber\\
 &\quad -\frac1{1-q}\int_{\R^N}\left(\rho^q-q\rho_*^{q-1}(\rho-\rho_*)-\rho_*^q\right)(x)\,\rd x.
\label{eq:diff_free_energy}
\end{align}
Let us distinguish cases here. If $M_*=0$, then $\int_{\R^N}\rho_*=1$ and the first term on the right side is certainly non-negative. If $M_*>0$, then the integral can have either sign, but the term vanishes since we have chosen $L_*=0$. Hence in both cases, the first term is non-negative. This is how the two conditions $L_*=0$ for $\int_{\R^N}\rho_*<1$ and $\int_{\R^N}\rho_*=1$ for $L_*>0$ appear to be necessary for minimizers. The other solutions $(\rho_L,1-\int_{\R^N}\rho_L)$ cannot be minimizers. On the other hand, by concavity of $a\mapsto a^q$ on $\R_+$ we have $a^q-qb^{q-1}(a-b)-b^q\leq0$ with equality if and only if $a=b$. Thus we have proved, as we claimed, that $\cF[\rho,M]\geq\cF[\rho_*,M_*]$ with equality if and only if $\rho=\rho_*$. This implies $M=M_*$ due to the mass constraint and concludes the proof of Theorem~\ref{main}.\qed

\begin{remark}[Concentration for $3\leq N\leq 5$]\label{rmk:min_infinite}
For $\frac{N-2}{N+2}<q\leq\frac{N}{N+4}$ we can still define $(\rho_*,M_*)$ by the same procedure. The right side of~(\ref{eq:diff_free_energy}) makes sense and is non-negative. Thus $(\rho_*,M_*)$ is a formal minimizer of the free energy, but the corresponding value is infinite: $\cF[\rho_*,M_*]=-\ii$. This is due to the first and third integrals in~(\ref{eq:free_energy_lambda_4_radial}) which diverge at large $x$. With this formal definition of a minimizer, we see that concentration indeed happens at $\lambda=4$ for all $N\geq3$, since $q_N(4)>\frac{N-2}{N+2}$ in this case. This fact can allow one to use relative free energy quantities as a tool to quantify the basin of attraction to equilibria similarly to the case of the Barenblatt solutions in the very fast diffusion range even if their second moment or their mass becomes infinite leading to infinite free energy, see \cite{CLMT02,LM03,BBDGV09}.
\end{remark}
\begin{remark}[Concentration is independent of mass]\label{rmk:independent_mass}
Notice that we have chosen to work with probability measures and with interaction potential $|x|^\lambda/\lambda$. Assuming a different mass normalization $\mu(\R^N)=m$ and an interaction potential $C|x|^\lambda$ with $C>0$, the minimizer $\mu_{*,m,C}$ of the free energy~(\ref{freeenergy}) is given by
$$
\begin{cases}
\mu_{*,m,C}=\gamma_1\mu_{*}(\gamma_2 x)\\
\gamma_1\gamma_2^{-N} = m\\
C\lambda \gamma_1^{3-q}\gamma_2^{-\lambda-2N}=m,
\end{cases}
$$
with $\mu_*$ the minimizer of the free energy~(\ref{freeenergy}) of mass 1. 
\end{remark}

%%%%%%%%%%%%%%%%%%%%%%%%%%%%%%%%%%%%%%%%%%%%%%%%%%%%%
%%%%%%%%%%%%%%%%%%%%%%%%%%%%%%%%%%%%%%%%%%%%%%%%%%%%%
\section{Numerical results for other values of $\lambda$}\label{sec:numerics}
%%%%%%%%%%%%%%%%%%%%%%%%%%%%%%%%%%%%%%%%%%%%%%%%%%%%%
%%%%%%%%%%%%%%%%%%%%%%%%%%%%%%%%%%%%%%%%%%%%%%%%%%%%%

For $\lambda=4$, we have shown in Lemma~\ref{solutions} that the mass $m(L)$ of the unique solution $\rho_L$ to the simplified Euler-Lagrange equation~(\ref{eq:EL_lambda_4}) is monotone decreasing with respect to the parameter $L$. Therefore, when $m(0)<1$ we can conclude that concentration occurs. We think that the same strategy applies to other values of $\lambda$. Namely, the idea is to look for a radial solution of the Euler-Lagrange equation~(\ref{eq:EL1}) \emph{imposing} $L=0$, and compute its mass $m(0)$. Should the latter be less than 1, we would have found a stationary state with concentration, which is a good candidate for being a global minimizer. If the mass is monotone with $L$ there exist in fact no solution of the Euler-Lagrange equation with mass one. If on the contrary $m(0)>1$, this proves that a minimizer has to have $L>0$ and thus it cannot display concentration, by~\cite[Prop.~14]{CDDFH19}.

To make this strategy work, we would have to show the existence and uniqueness of solutions to~(\ref{eq:EL1}) for every $L\geq0$ and prove that the corresponding mass $m(L)$ is decreasing. We are not able to show this fact analytically. However, we have numerically investigated this question thoroughly for a large range of $\lambda$. The numerical evidence supports the conjecture that these claims hold. We present here our numerical results, dealing first with the simpler case of even integers. The python scripts and their outputs can be found in the GitHub repository `Concentration'~\cite{CarDelFraLew-GitHub} together with a list of the cases investigated.

\subsection{The case of even integers}\label{sec:even_lmb}
When $\lambda=2n$ is an even integer, we can expand $|x-y|^{2n}$ in terms of a polynomial in $x$ and $y$. Using the radial symmetry of $\rho_*$, we obtain that only the even terms contribute to the convolution $\rho_*\ast|\cdot|^\lambda$ in the Euler-Lagrange equation~(\ref{eq:EL1}): 
\begin{equation}\label{eq:evenintegers}
\begin{array}{rcl}
\displaystyle\frac{q}{1-q}\rho_*^{q-1}(x)&\!=\!&\displaystyle{
\int_{\R^N}|x-y|^\lambda \rho_*(y) \rd y+M_*|x|^\lambda-\int_{\R^N}|y|^\lambda\rho_*(y)\,\rd y\;+L}\\
&\!=\!&\displaystyle|x|^{\lambda}\!+\!\sum_{i=1}^{n-1}c^\lambda_i\left(\int_{\R^N} |y|^{2n-2i}\rho_*(y)\,\rd y\right) \!|x|^{2i}+L,
\end{array}
\end{equation}
where we have used the fact that the coefficient of $|x|^\lambda$ simplifies due to the mass constraint, and where
$c_i^\lambda$ are positive coefficients that depend on $\lambda$ and $N$. The values of $c_i^\lambda$ can be computed explicitly and are provided in the appendix for $\lambda\in\{2,4,6,8,10\}$ for completeness. For instance, we have $c^4_1=2+4/N$, $c^6_1=c^6_2=3+12/N$, etc. Note that the unknown mass $M_*$ has disappeared from the equation but it of course still appears in the mass constraint~(\ref{mass}).

The new form~(\ref{eq:evenintegers}) of our Euler-Lagrange equation implies that we can restrict our attention to densities 
\begin{equation}\label{eq:parametric}
\rho_{\beta,L}(x)=\left(\frac{q}{1-q}\right)^{\frac1{1-q}}\left(|x|^{\lambda}+\sum_{i=1}^{n-1}\beta_i |x|^{2i}+L\right)^{-\frac{1}{1-q}}\!\!\!\!,
\end{equation}
with $\beta_i,L\geq0$. The equation~(\ref{eq:evenintegers}) then reduces to a system of $n-1$ nonlinear equations for the parameters $\beta_i$:
\begin{equation}\label{eq:betas}
\beta_i=c_i^\lambda \int_{\R^N}|y|^{2n-2i}\rho_{\beta,L}(y)\;\rd y=:F_{i,L}(\beta_1,...,\beta_{n-1}),
\end{equation}
with $i=1,\dots,n-1$. The gradient of $F_{i,L}$ equals
$$
\frac{\partial}{\partial \beta_j} F_{i,L}(\beta_1,...,\beta_{n-1})=-\frac{c_i^\lambda}{q} \int_{\R^N}|y|^{2n-2i+2j}\rho_{\beta,L}(y)^{2-q}\,\rd y.
$$
We have implemented in Python an algorithm which solves the nonlinear equations~(\ref{eq:betas}). We used the BFGS algorithm on the auxiliary functional 
$$I(\beta_1,...,\beta_{n-1}):=\sum_{i=1}^{n-1}\big(\beta_i-F_{i,L}(\beta_1,...,\beta_{n-1})\big)^2,$$ 
with the integrals computed by the quadrature method, both from the Python library \emph{SciPy}~\cite{SciPy}.

In Figure~\ref{fig:even_lmb} we provide the result of the computation of the mass $m(0)$ for $\lambda\in\{4,6,8,10\}$ as a function of the parameter
\begin{equation}
\alpha:=\frac{2N-q(2N+\lambda)}{N(1-q)}.
\label{eq:alpha_q}
\end{equation}
This parameter (already used in~\cite{CDDFH19}) is useful since the critical value $q=\frac{N}{N+\lambda}$ simply becomes $\alpha=1$ and thus does not depend on $N$ and $\lambda$. The interval of interest is then $\alpha\in(0,1)$ but, similarly as in Lemma~\ref{solutions}, one can in fact go up to $\alpha=2-\frac\lambda4+\frac\lambda{2N}$ which corresponds to $q=\frac{N-2}{N+2}$. In the case $1\le \alpha$, we know that the free energy is not bounded below, still the constructed candidate is a minimizer in the relative free energy sense, see Remark~\ref{rmk:min_infinite}. From the figure we see that the mass $m(0)$ is smaller than $1$ for some $\alpha<1$ in all dimensions $N\geq4$ for $\lambda\in\{6,8,10\}$. We also recover the results of Theorem~\ref{main} at $\lambda=4$. In dimension $N=3$ the mass barely misses 1 at $\lambda=8$ but concentration happens at $\lambda=10$. In Table~\ref{tab:Q} we provide an approximation of the corresponding critical $q_N(d)$ below which concentration occurs, which clearly illustrates the monotonicity in $N$ and $\lambda$.  

We have also found numerically that the mass $m(L)$ of the solution was always decreasing with $L$ as pointed out earlier. Should this be confirmed, then we would have a clear picture of the situation for $\lambda$ an even integer. The final conclusion would be that concentration occurs for all $N\geq3$ starting at $\lambda=10$ and for all $N\geq4$ for $\lambda\in\{6,8\}$. We have already proved in Theorem~\ref{main} that it happens for all  $N\geq6$ for $\lambda=4$, and it was shown in~\cite{CDDFH19} that there is never concentration for $\lambda=2$. Of course we expect the true critical values of $\lambda$ to be in between these special cases of even integers. For instance, we think that concentration happens in $N=3$ for $\lambda$ just slightly above $\lambda=8$, as is suggested by Figure~\ref{fig:even_lmb}. 

\begin{figure*}[t]
\centering
\begin{tabular}{cc}
$\lambda=4$& $\lambda=6$ \\[-0.1cm]
\includegraphics[width=7cm]{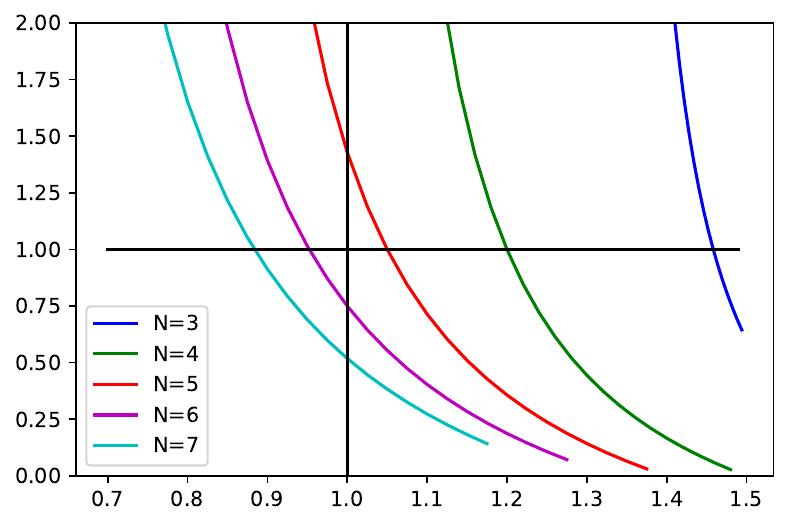}&
\includegraphics[width=7cm]{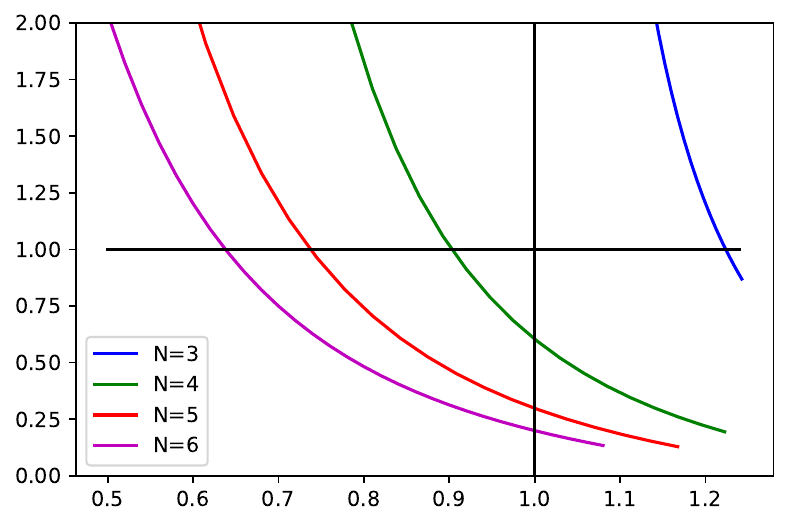}\\[0.2cm]
$\lambda=8$& $\lambda=10$\\[-0.1cm]
\includegraphics[width=7cm]{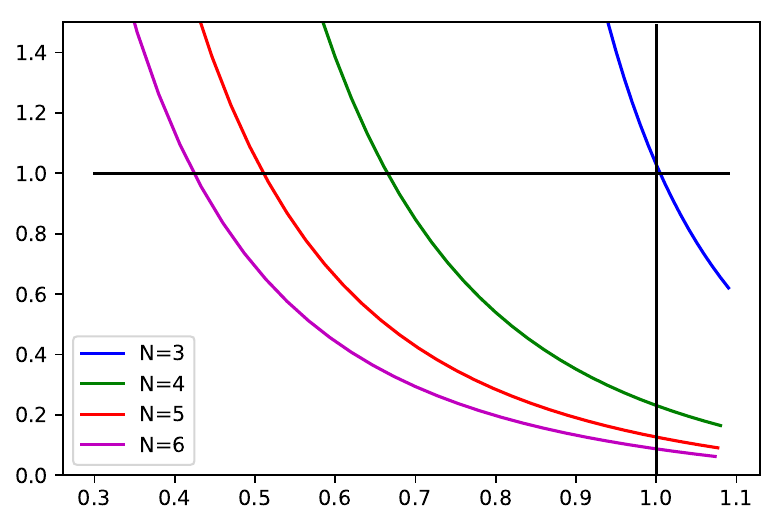}&
\includegraphics[width=7cm]{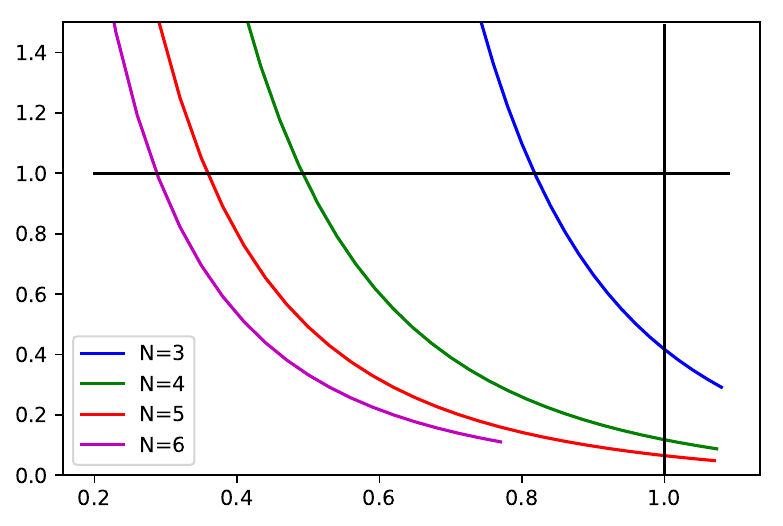}
\end{tabular}
\caption{Mass $m(0)$ of the solution of~(\ref{eq:betas}) found at $L=0$, in terms of the parameter $\alpha$ in~(\ref{eq:alpha_q}), for the indicated values of $\lambda$. \label{fig:even_lmb}}
\end{figure*}

\begin{table}[h]
\centering
\begin{tabular}{c|c|c|c|c}
& $N=3$ & $N=4$ & $N=5$ & $N=6$\\
\hline
$\lambda=4$ & X & X & X & \!\!$\frac{11}{18}\simeq 0.61\;(0.95)$\!\!\!\\
\hline
$\lambda=6$ & X & $0.42\;(0.90)$ & $0.52\;(0.72)$& $0.58\;(0.62)$\\
\hline
$\lambda=8$ & X & $0.40\;(0.66)$& $0.48\;(0.50)$& $0.54\;(0.41)$\\
\hline
$\lambda=10$ & $0.26\;(0.81)$ & $0.38\;(0.49)$& $0.45\;(0.35)$ &$0.51\;(0.29)$ \\
\end{tabular}

\medskip

\caption{Numerical value of the critical diffusion exponent $q_{N}(\lambda)$ below which concentration occurs. The corresponding $\alpha_N(\lambda)$ in~(\ref{eq:alpha_q}) is provided in parenthesis (these are the points at which the curves of Figure~\ref{fig:even_lmb} cross the horizontal axis $m=1$). The value of $q_N(\lambda)$ for $\lambda=4$ and $N=6$ is from Theorem~\ref{main}. An `X' means that concentration does not occur for minimizers of the free energy in this combination of parameters. \label{tab:Q}}
\end{table}

\subsection{The general case}
For $\lambda\notin2\N$, the convolution $\rho\ast|\cdot|^\lambda$ is not a polynomial in $|x|^2$, hence the Euler-Lagrange equation cannot be reduced to a nonlinear equation involving finitely many parameters as in~(\ref{eq:betas}). However, we can still restrict our attention to radial functions and write the convolution in the form
$$\rho\ast|\cdot|^\lambda(x)=|\bS^{N-1}|\int_0^\infty K_{N,\lambda}(|x|,s)\rho(s)s^{d-1}\rd s,$$
with the slight abuse of notation $\rho(|x|):=\rho(x)$, where the kernel $K_{N,\lambda}(r,s)$ is the spherical average
$$K_{N,\lambda}(r,s):=\frac{1}{|\bS^{N-1}|^2}\iint_{\bS^{N-1}\times \bS^{N-1}}|r\omega-s\omega'|^\lambda\,\rd\omega\,\rd\omega'.$$
Its explicit formula in terms of Hypergeometric functions is provided later in the appendix for completeness. Our goal is to solve the nonlinear equation~(\ref{eq:EL1}) at $L=0$. A Taylor expansion as in~\cite[Lemma 18]{CDDFH19} gives 
$$
\rho(r)=\begin{cases}
Cr^{-\frac{2}{1-q}}(1+o(1))& \text{for $r\to0$,}\\
\left(\frac{q}{1-q}\right)^{\frac1{1-q}}r^{-\frac{\lambda}{1-q}}(1+o(1))&\text{for $r\to\ii$}.
\end{cases}
$$
This prompts us to make the ansatz
\begin{equation}
\rho_f(r)^{q-1}= r^2(1+r)^{\lambda-2}f(r),
 \label{eq:def_f}
\end{equation}
where $f$ is a positive continuous function on $\R_+$ tending to a positive limit at 0 and infinity. The Euler-Lagrange equation~(\ref{eq:EL1}) can be expressed in terms of $f$ as
\begin{align}
f(r)=\Phi(f)(r):=&\frac{1-q}{q} \left\{|\bS^{N-1}| \int_0^\ii \frac{C_{N,\lambda,q}(r,s)}{f(s)^{\frac{1}{1-q}}}\rd s\right.\nonumber \\
&\qquad
+\left.\left(1-|\bS^{N-1}|\int_0^\ii\frac{s^{N-1-\frac2{1-q}}}{(1+s)^{\frac{\lambda-2}{1-q}}f(s)^{\frac{1}{1-q}}}\rd s\right)\frac{r^{\lambda-2}}{(1+r)^{\lambda-2}}\right\},
\label{eq:generalEL}
\end{align}
where
$$
C_{N,\lambda,q}(r,s)=\frac{K_{N,\lambda}(r,s)-s^\lambda}{r^2(1+r)^{\lambda-2}}\frac{s^{N-1-\frac{2}{1-q}}}{(1+s)^{\frac{\lambda-2}{1-q}}}.
$$
We used again the mass condition~(\ref{mass}) to remove $M$ from the equation. Our goal is to find a numerical approximation of the solution $f$ to~(\ref{eq:generalEL}). We approximate $f$ by 
\begin{equation}
f(s)=\left|P\left(\frac{r}{1+r}\right)\right|^2
\label{eq:discretization}
\end{equation}
where $P$ is a (complex) polynomial of finite degree $d$. This choice is motivated by the fact that when $\lambda=2n\in2\N$, the exact solution takes exactly this form with $d=n-1$. For $\lambda\notin2\N$ there will probably be no solution to the nonlinear equation $f=\Phi(f)$ in the class~(\ref{eq:discretization}) and we rather aim at minimizing a certain distance between $f$ and $\Phi(f)$. In our algorithm we minimized the $L^2$ (square) distance 
\begin{equation}
 \int_0^\ii\frac{s^{N-1-\frac2{1-q}}}{(1+s)^{\frac{\lambda-2}{1-q}}}\left(f(s)^{-\frac{1}{1-q}}-\Phi(f)(s)^{-\frac{1}{1-q}}\right)^2\rd s
 \label{eq:L2_distance}
\end{equation}
with respect to the coefficients of the polynomial $P$. By Jensen's inequality, the $L^2$ norm in~(\ref{eq:L2_distance}) controls the $L^1$ norm of the corresponding densities in~(\ref{eq:def_f})
\begin{align}
\norm{\rho_f-\rho_{\Phi(f)}}_{L^1(\R^N)}=&  \left(\frac{q}{1-q}\right)^{\frac1{1-q}}|\bS^{N-1}|\\
&\times\int_0^\ii\frac{s^{N-1-\frac2{1-q}}}{(1+s)^{\frac{\lambda-2}{1-q}}}\left|f(s)^{-\frac{1}{1-q}}-\Phi(f)(s)^{-\frac{1}{1-q}}\right|\rd s
 \label{eq:L1_distance}
\end{align}
and it is a smoother function of the polynomial coefficients.

We used again the BFGS method of \emph{SciPy} to minimize~(\ref{eq:L2_distance}) with respect to the coefficients of the polynomial $P$ in~(\ref{eq:discretization}). This provides an approximate solution to~(\ref{eq:generalEL}). We then computed the mass of this solution in terms of $q$ and obtain curves which look very much like those of Figure~\ref{fig:even_lmb}. The point at which the mass crosses 1 provides an approximation of the critical exponent $q_N(\lambda)$ below which condensation happens. 

We provide in Figure~\ref{fig:critical} the result of the calculation of $q_N(\lambda)$ in dimension $N=5$. The curve suggests that concentration starts to appear for $\lambda$ slightly above $4$. This curve was obtained by taking for $P$ a complex polynomial of degree 10 and discretizing the integrals in~(\ref{eq:L2_distance}) and~(\ref{eq:generalEL}) by a simple Riemann sum with $1000$ regularly spaced points on $[0,20]$. With this choice of discretization parameters, the  $L^1$ error~(\ref{eq:L1_distance}) was found to be at most of the order $10^{-5}$ for all the points of the curve.

\begin{figure}%[tbhp]
\centering
\includegraphics[width=8.8cm]{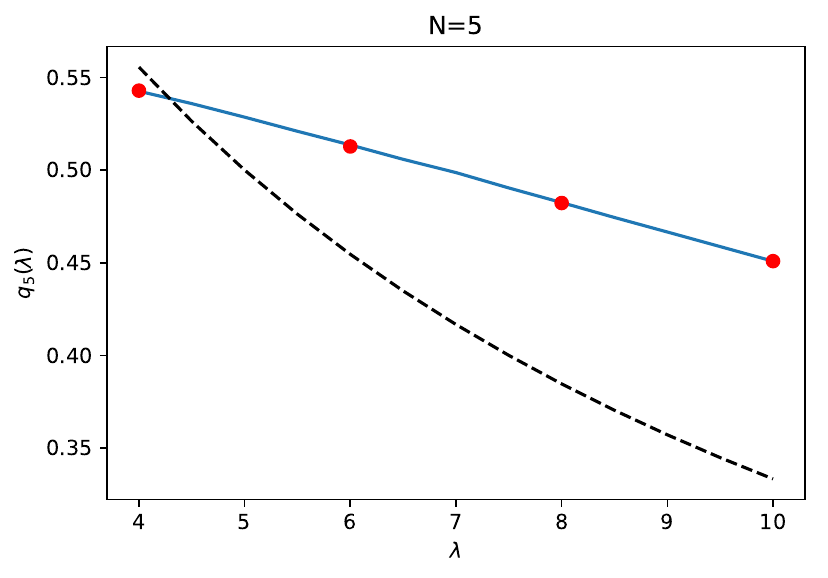}
\caption{Numerical value of $q_5(\lambda)$ below which condensation occurs. The dashed line is the curve $\lambda\mapsto\frac{N}{N+\lambda}$ which is a lower bound to $q$. The dots are the points computed with the algorithm from Section~\ref{sec:even_lmb} when $\lambda$ is an even integer. This suggests that concentration happens in dimension $N=5$ for $\lambda$ slightly above $4$. \label{fig:critical}}
\end{figure}

\section{Discussion}

We have studied free energy minimizers of the nonlinear aggregation-diffusion equation~(\ref{FP}) and found that these can have part of their mass concentrated in a Dirac delta at the origin. This complements the literature on variants of the KS model and their possible blow-up behavior, see \cite{bellomo2020chemotaxis} for recent developments. We rigorously proved this property in the case of the quartic attraction $\lambda=4$ in all space dimensions $N\geq6$ and provided numerical results which suggest the same phenomenon in all space dimensions $N\geq3$ for $\lambda=10$. Our analysis is based on the study of the integral $m(L)$ of (the regular part of) solutions to the Euler-Lagrange equation with Lagrange multiplier $L$, and mass concentration means that $m(0)<1$. Numerical evidence supports that $m(L)$ is decreasing with respect to $L$. It is an interesting problem to show that these candidates are the unique minimizers of the free energy (\ref{freeenergy}) for $\lambda>4$. Unlike the results of \cite{BF83}, which proved that the Dirac delta is an unstable equilibrium for the fast diffusion equation with $0<q<\frac{N-2}{N}$, we expect that the minimizers exhibiting concentration are attractors of the dynamics (\ref{FP}). It is an interesting open problem if an appropriate notion of solution to (\ref{FP}) can develop singularities in finite or infinite time. A result in this direction in the case of an external potential instead of the interaction potential as in (\ref{FP}) has recently been obtained in \cite{CGV21}.

\section*{Appendix: Computation of $K_{N,\lambda}(r,s)$}
Passing to radial coordinates, we obtain 
\begin{align*}
K_{N,\lambda}(r,s)%&=\frac1{|\bS^{N-1}|^2}\iint_{\bS^{N-1}\times\bS^{N-1}}|r\omega-s\omega'|^\lambda\rd \omega\,\rd\omega'\\
% &=|\bS^{d-1}|^{-1}\int_{\bS^{N-1}}\left(r^2+s^2-2rs\omega_1\right)^{\frac\lambda2}\rd \omega\\
&=\frac{|\bS^{N-2}|}{|\bS^{N-1}|}\int_{0}^\pi \!\!\left(r^2+s^2-2rs\cos\phi\right)^{\frac\lambda2}\sin^{N-2}(\phi)\,\rd\phi\\
&=\frac{|\bS^{N-2}|}{|\bS^{N-1}|}\int_{-1}^1 \left(r^2+s^2-2rst\right)^{\frac\lambda2}(1-t^2)^{\frac{N-3}{2}}\,\rd t.
\end{align*}
In dimension $N=3$ we have the explicit formula
$$K_{3,\lambda}(r,s)=\frac1{2(\lambda+2)}\frac{(r+s)^{\lambda+2}-|r-s|^{\lambda+2}}{rs}$$
while, in other dimensions, this can be expressed as
$$
K_{N,\lambda}(r,s)=\left(r^2+s^2\right)^{\frac\lambda2}\,_2F_1\left(\frac{2-\lambda}{4},-\frac{\lambda}{4},\frac{N}{2},\frac{4r^2s^2}{(r^2+s^2)^2}\right)
$$
where $_2F_1$ is the Hypergeometric function. In the case of even integers we find
\begin{align*}
K_{N,2}(r,s)=&r^2+s^2,\\
K_{N,4}(r,s)=&r^4+s^4+\left(2+\frac4N\right)r^2s^2,\\
K_{N,6}(r,s)=&r^6 + s^6+ \left(3+\frac{12}N\right) (r^4 s^2 +  r^2 s^4) ,\\
K_{N,8}(r,s)=&r^8+s^8+\frac{32+48/N+4N}{2+N}(r^6 s^2+s^6r^2)\\
&\qquad +\frac{60+144/N+6N}{2+N}r^4 s^4, \\
K_{N,10}(r,s)=&r^{10}+s^{10}+\frac{50+80/N+5N}{2+N}(r^8 s^2+s^8r^2)\\
&\qquad+\frac{140+480/N+10N}{2+N}(r^6 s^4+s^6 r^4). 
\end{align*}
This provides the coefficients $c_i^\lambda$ appearing in~(\ref{eq:betas}).

% \subsection*{Supporting Information (SI)}
% 
% Authors should submit SI as a single separate PDF file, combining all text, figures, tables, movie legends, and SI references.  PNAS will publish SI uncomposed, as the authors have provided it.  Additional details can be found here: \href{http://www.pnas.org/page/authors/journal-policies}{policy on SI}.  For SI formatting instructions click \href{https://www.pnascentral.org/cgi-bin/main.plex?form_type=display_auth_si_instructions}{here}.  The PNAS Overleaf SI template can be found \href{https://www.overleaf.com/latex/templates/pnas-template-for-supplementary-information/wqfsfqwyjtsd}{here}.  Refer to the SI Appendix in the manuscript at an appropriate point in the text. Number supporting figures and tables starting with S1, S2, etc.
% 
% Authors who place detailed materials and methods in an SI Appendix must provide sufficient detail in the main text methods to enable a reader to follow the logic of the procedures and results and also must reference the SI methods. If a paper is fundamentally a study of a new method or technique, then the methods must be described completely in the main text.

\section*{Acknowledgements}

The authors would like to thank Jean Dolbeault, David G\'omez-Castro, Juan Luis V\'az\-quez and an anonymous referee for pointing us out reference \cite{BF83} and fruitful comments. This project has received funding from the European Research Council (ERC) under the European Union's Horizon 2020 research and innovation program (Advanced Grant Nonlocal-CPD 883363 of J.A.C. and Consolidator Grant MDFT 725528 of M.L.). M.G.D. was partially supported by CNPq-Brazil (\#308800/2019-2) and Instituto Serrapilheira. R.L.F. was partially supported by the U.S. National Science Foundation through grants DMS-1363432 and DMS-1954995 and through Germany’s Excellence Strategy EXC-2111-390814868.

% Bibliography
\bibliographystyle{abbrv}
\bibliography{CDDFH}

{% additional braces for segregating \footnotesize
  \bigskip
  \footnotesize
  
  J.A.~Carrillo
  
  \textsc{Mathematical Institute, University of Oxford, Oxford OX2 6GG, UK}\par\nopagebreak
  \textit{E-mail address:} \texttt{carrillo@maths.ox.ac.uk}

  \medskip

 M.G.~Delgadino
 
 \textsc{Mathematical Institute, University of Oxford, Oxford OX2 6GG, UK}\par\nopagebreak
  \textit{E-mail address:} \texttt{matias.delgadino@maths.ox.ac.uk}
  
    \textsc{Pontifical Catholic University of Rio de Janeiro, 38097 RJ, Brazil}\par\nopagebreak
  \textit{E-mail address:} \texttt{matias.delgadino@mat.puc-rio.br }(On Leave)

\medskip

  R.L.~Frank
  
  \textsc{Department  of Mathematics, California Institute of Technology, Pasadena, CA 91125, USA}\par\nopagebreak
  \textit{E-mail address:} \texttt{rlfrank@caltech.edu}

  \textsc{Mathematisches Institut, Ludwig-Maximilans Universit\"at M\"unchen, Theresienstr.~39, 80333 M\"unchen, Germany, and Munich Center for Quantum Science and Technology, Schellingstr.~4, 80799 M\"unchen, Germany}\par\nopagebreak
  \textit{E-mail address:} \texttt{r.frank@lmu.de}
  
  \medskip
  
   M.~Lewin
  
  \textsc{CNRS \& CEREMADE, University Paris-Dauphine, PSL University, 75 016 Paris, France}\par\nopagebreak
  \textit{E-mail address:}  \texttt{mathieu.lewin@math.cnrs.fr}

}

% \affil[a]{Mathematical Institute, University of Oxford, Oxford OX2 6GG, UK}
% \affil[b]{Pontifical Catholic University of Rio de Janeiro, 38097 RJ, Brazil}
% \affil[c]{Department  of Mathematics, California Institute of Technology, Pasadena, CA 91125, USA}
% \affil[d]{Mathematisches Institut, Ludwig-Maximilans Universit\"at M\"unchen, Theresinstr.~39, 80333 M\"unchen, Germany}
% \affil[e]{CNRS \& CEREMADE, University Paris-Dauphine, PSL University, 75 016 Paris, France}

\end{document}